\documentclass[a4paper,12pt,twoside]{amsart}
\usepackage{amsmath}
\usepackage{amsfonts}
\usepackage{amssymb}
\usepackage{amsthm}
\usepackage[english]{babel}
\usepackage{enumerate}
\usepackage{color}
\usepackage[left=2.5cm,right=2.5cm,top=2cm,bottom=2.5cm]{geometry}

\usepackage[svgnames,hyperref]{xcolor}
\newcommand{\mycolor}{Navy}
\usepackage[pdfauthor={Son H. Do}, colorlinks, linktocpage, citecolor = \mycolor, linkcolor = \mycolor, urlcolor = \mycolor]{hyperref}
\newtheorem{The}{Theorem}
\newtheorem{Lem}[The]{Lemma}
\newtheorem{Cor}[The]{Corollary}
\newtheorem{Prop}[The]{Proposition}

\newtheorem{Def}[The]{Definition}

\newcommand{\C}{\mathbb{C}}
\newcommand{\R}{\mathbb{R}}

\newcommand{\dt}{\partial_t}
\newcommand{\e}{\epsilon}
\newcommand{\p}{\psi}

\begin{document}
 \title[Parabolic complex Hessian type equations]
{Viscosity solutions to  Parabolic complex Hessian type equations} 
\setcounter{tocdepth}{1}
\author{Hoang-Son Do} 
\address{Institute of Mathematics \\ Vietnam Academy of Science and Technology \\18
	Hoang Quoc Viet \\Hanoi \\Vietnam}
\email{hoangson.do.vn@gmail.com, dhson@math.ac.vn}
\date{\today
	 The author was supported by Vietnam Academy of Science and Technology under grant number CT0000.07/21-22.}
\maketitle
\begin{abstract}
	In this paper, we show the existence and uniqueness of viscosity solution to the Cauchy-Dirichlet problem for a class of fully nonlinear parabolic equations.  This extends recent results of Eyssidieux-Guedj-Zeriahi in \cite{EGZ15b}.
\end{abstract}
\tableofcontents
\section{Introduction}
	Let $\Gamma\varsubsetneq\R^n$ be an open, convex, symmetric cone with vertex at $0$ such that $\Gamma_n\subseteq\Gamma\subseteq\Gamma_1$,
where $\Gamma_k$ is the set of all $x\in\R^n$ such that the $l$-th elementary symmetric sum $\sigma_l(x)>0$  for every $1\leq l\leq k$.
Let $f:\overline{\Gamma}\rightarrow [0, \infty)$
be a symmetric, concave function such that
$f$ is  strictly increasing in each variable and $f|_{\partial\Gamma}=0$ (and then $f|_{\Gamma}>0$).
We define $F:\mathcal{H}^n\rightarrow [-\infty, \infty)$ by
\begin{equation}
	F(H)=\begin{cases}
		f(\lambda(H))\quad\mbox{if}\quad H\in \overline{M(\Gamma, n)},\\
		-\infty\quad\mbox{if}\quad H\in \mathcal{H}^n\setminus\overline{M(\Gamma, n)}.
	\end{cases}
\end{equation}
where $\mathcal{H}^n$ is the set of all $n\times n$ Hermitian matrices and
$M(\Gamma, n)$ is the subset of $\mathcal{H}^n$ containing
matrices $H$ with the eigenvalues $\lambda (H)=(\lambda_1,...,\lambda_n)\in\Gamma$.
 The conditions on $f$ imply that $F$ is concave
on $\overline{M(\Gamma, n)}$ (see \cite{CNS}) and
\begin{equation}\label{F is increasing}
	F(M+N)>F(M),
\end{equation}
for every $M\in M(\Gamma, n)$ and for each positive semidefinite  matrix $N\neq 0$.\\

 Let $\Omega\subset \mathbb{C}^n$ be a bounded domain and $T>0$. We consider the Cauchy-Dirichlet
 problem
\begin{equation}\label{PHE dirichlet}
	\begin{cases}
		F(Hu)=e^{\dt u+G(t, z, u)}g(z)\qquad\mbox{ in }\qquad \Omega_T,\\
		u=\varphi\qquad\mbox{ in }\qquad [0, T)\times\partial\Omega,\\
		u(0, z)=u_0(z)\qquad\mbox{ in }\quad\overline{\Omega},
	\end{cases}
\end{equation}
where 
\begin{itemize}
	\item  $\Omega_T=(0, T)\times\Omega$.
	\item $Hu$ is the complex Hessian of $u$.
	\item $G(t, z, r)$ is a continuous function in $[0, T]\times\overline{\Omega}\times\mathbb{R}$ 
	which is non-decreasing in the last variable.
	\item  $g\geq 0$ is a continuous, bounded function in  $\Omega$.
	\item $\varphi(t, z)$ is continuous in $[0, T]\times\partial\Omega$.
	\item $u_0(z)$ is continuous in $\overline{\Omega}$ and $\Gamma$-subharmonic in $\Omega$ with $u_0(z)=\varphi(0, z)$ for every $z\in\partial\Omega$.
\end{itemize}
In the case of Parabolic complex Monge-Amp\`ere equation (i.e., $f(x)=(\sigma_n(x))^{1/n}$ and $\Gamma=\Gamma_n$), Eyssidieux-Guedj-Zeriahi \cite{EGZ15b} show that the Cauchy-Dirichlet problem has a unique
viscosity solution provided $\Omega$ is strictly pseudoconvex and $(u_0, g)$ is admissible. In this note, we verify that the same result holds in the general case.

We say that the pair $(u_0, g)$ is admissible if for all $\epsilon>0$, there
exist $u_{\epsilon}\in C(\bar{\Omega})$ and $C_{\epsilon}>0$ such that $u_0\leq u_{\epsilon}\leq u_0+\epsilon$ and $F(H u_{\epsilon})\leq 
e^{C_{\epsilon}}g(z)$ in the viscosity sense. Our main result is as follows:
\begin{The}
	Assume $\Omega$ is a strictly $\Gamma$-pseudoconvex domain and $f$ satisfies
	\begin{center}
		$\lim\limits_{R\to\infty}f(R,...,R)=\infty.$
	\end{center}
	Then, the Cauchy-Dirichlet problem \eqref{PHE dirichlet} admits a unique viscosity solution iff $(u_0, g)$ is admissible. Moreover, if $u$ is a viscosity (sub-)solution of \eqref{PHE dirichlet} then $z\longmapsto u(t, z)$ is $\Gamma$-subharmonic in $\Omega$ for every $t\in (0, T)$.
\end{The}
The note is organized as follows: in the section \ref{sec pre}, we recall backgrounds on
viscosity sub/super-solutions, $\Gamma$-subharmonic functions and strictly $\Gamma$-pseudoconvex domains;
in the section \ref{sec reg}, we prove some lemmas about approximating a sub/super-solution
by a sequence of sub/super-solutions which are Lipschitz in $t$; the $\Gamma$-subharmonicity of subsolutions will be shown in Section \ref{sec subharmonic}; the comparison principle and the Perron method will be presented
in Section \ref{sec compa}; in the last section, we prove the existence and uniqueness of solution.
\section{Preliminaries}\label{sec pre}
In this section, we recall the definitions and some properties of viscosity sub/super-solutions
as well as $\Gamma$-subharmonicity. The reader can find more
details in \cite{Jen88}, \cite{Ish89}, \cite{IL90}, \cite{CIL92}, \cite{EGZ15b} and \cite{DDT}.
\subsection{Viscosity concepts}
\begin{Def}(Test functions) Let $w : \Omega_T \rightarrow \mathbb{R}$ be a function  and let $(t_0,z_0) \in \Omega_T$. An upper test function (resp., a lower test function) for  $w$ at $(t_0, z_0)$ is a function $q$ is $C^{(1, 2)}$-smooth (i.e. $q$ is $C^1$-smooth in $t$ and $C^2$-smooth in $z$) in a neighbourhood of $(t_0, z_0)$ such that $w(t_0, z_0)=q(t_0, z_0)$ and $\omega \leq q$ (resp., $w\geq q$) 
in a neighbourhood of $(t_0, z_0)$.
\end{Def}

\begin{Def}
	\begin{enumerate}[(i)]	
		\item  A function $u \in USC(\Omega_T)$ is said to be a (viscosity) subsolution of the parabolic equation
		\begin{equation}\label{PHE def}
				F(Hw)=e^{\dt w+G(t, z, w)}g(z),
		\end{equation}
		in $\Omega_T$ if for every $(t_0, z_0) \in \Omega_T$ and for each upper test function $q$ of $u$ at $(t_0, z_0) $, we have
		$$
		F(Hq(t_0, z))\mid_{z=z_0}  \geq e^{\partial_t q (t_0,z_0) + G(t_0,z_0,q (t_0,z_0))} g (z_0).
		$$
		In this case, we also say that
		$$F(Hu)  \geq e^{\partial_t u (t,z) + G(t,z,u (t,z))} g(z),$$ 
		in the viscosity sense in $\Omega_T$.
		
		A function $u \in USC(\left[ 0, T\right) \times \overline{\Omega} )$ said to be a subsolution of  the Cauchy - Dirichlet problem \eqref{PHE dirichlet} if $u$ is a subsolution of \eqref{PHE def} satisfying $u \leq \varphi$ in $\left[ 0, T\right) \times \partial \Omega$ and $u(0, z) \leq u_0(z)$ for every $z \in \Omega$.
		
		\item  A function $v \in LSC(\Omega_T)$ is said to be a (viscosity) supersolution of \eqref{PHE def} in $\Omega_T$ if for every $(t_0, z_0) \in \Omega_T$ and for each lower test function $q$ for $u$ at $(t_0, z_0) $, we have
		$$
			F(Hq(t_0, z))\mid_{z=z_0} \leq e^{\partial_{t} q (t_0,z_0) + G(t_0,z_0,q (t_0,z_0))} g (z_0).
		$$
		In this case, we also say that
		$$F(Hv) \leq e^{\partial_t v (t,z) + G(t,z,v(t,z))} g(z),$$
		in the viscosity sense in $\Omega_T$.
		
		A function $v \in LSC(\left[ 0, T\right) \times \overline{\Omega} )$ is said to be a (viscosity) supersolution of the Cauchy - Dirichlet problem \eqref{PHE dirichlet} if $v$ is  a  supersolution of \eqref{PHE def} satisfying $v \geq \varphi$ in $\left[ 0, T\right) \times \partial \Omega$ and $v(0, z) \geq u_0(z)$ for every $z \in \Omega$.
		
		\item  A function $u: \Omega_T \rightarrow \mathbb{R}$ is said to be a (viscosity) solution of \eqref{PHE def} (resp., \eqref{PHE dirichlet}) if it is a subsolution and a supersolution of \eqref{PHE def} (resp., \eqref{PHE dirichlet}).
	\end{enumerate}
\end{Def}
The following lemma is directly deduced from the definition of subsolution:
\begin{Lem}
	If $u_1, u_2$ are subsolutions of \eqref{PHE def} in $\Omega_T$ then $\max\{u_1, u_2\}$ is a
	subsolution of \eqref{PHE def} in $\Omega_T$.
\end{Lem}
The following (classical but non-trivial) result is a useful tool in the viscosity theory (see \cite{CIL92, IS12}):
\begin{Lem}\label{lem inf sup}
	Let $(u_{\tau})$ be a family of real-valued functions in  $\Omega_T$. 
	
	1. Assume that $\tau$, $u_{\tau}$ is a subsolution of 
	\begin{equation}\label{eq Lem liminf limsup}
		F(Hw)=e^{ \partial_t w + F(t,z,w)} g(z) ,
	\end{equation}
	in  $\Omega_T$ for every $\tau$ and $\sup_{\tau}u_{\tau}$ is bounded from above.
	Then
	$\overline u= (\sup_{\tau} u_{\tau})^*$ is a subsolution of
	 \eqref{eq Lem liminf limsup}
	in $\Omega_T$. 
	
	2. Assume that $u_{\tau}$ is a supersolution of \eqref{eq Lem liminf limsup} in $\Omega_T$ for every $\tau$
	and $\inf_{\tau}u_{\tau}$ is bounded from below. Then
	$\underline u = (\inf_{\tau} u_{\tau})_*$
	is a supersolution of  \eqref{eq Lem liminf limsup} in $\Omega_T$.
	
	3. If $\tau\in \mathbb{N}$ then 1. and 2. hold for 	$\overline u= (\limsup\limits_{\tau\to\infty} u_{\tau})^*$ and $\underline u = (\liminf\limits_{\tau\to\infty} u_{\tau})_*$.
\end{Lem}
\subsection{The parabolic Jensen-Ishii’s maximum principle}
 Denote by $\mathcal{S}_{2n}$ the space of all $2n\times 2n$ symmetric matrices.
For each function  $ u : \Omega_T \longrightarrow \R$ and for every $(t_0,z_0) \in \Omega_T$, 
we define by 	$\mathcal P^{2,+} u (t_0,z_0)$ the set of $(\tau, p, Q)\in\R\times \R^{2n}\times\mathcal{S}_{2n}$ satisfying
\begin{equation} 
	u (t,z)  \leq u (t_0,z_0) +  \tau (t-t_0) + o (\vert t-t_0\vert) 
	+ \langle p, z - z_0\rangle + \frac{1}{2} \langle Q (z - z_0), z-z_0\rangle   
	+ o (\vert z - z_0\vert^2),
\end{equation}
and denote by $\bar{\mathcal P}^{2,+} u (t_0,z_0)$ the set of 
$(\tau, p, Q)\in\R\times \R^{2n}\times\mathcal{S}_{2n}$ satisfying: $\exists (t_m, z_m)\rightarrow (t_0, z_0)$
and $(\tau_m, p_m, Q_m)\in \mathcal P^{2,+} u (t_0,z_0)$ such that $(\tau_m, p_m, Q_m)\rightarrow (\tau, p, Q)$ and
$u(t_m, z_m)\rightarrow u(t_0, z_0)  \}.$\\

We define in the same way  the  sets  $\mathcal P^{2,-} u (t_0,z_0)$ and $\bar{\mathcal P}^{2,-} u (t_0,z_0)$
by 
$$
\mathcal P^{2,-} u (t_0,z_0) = - \mathcal P^{2,+} (-u) (t_0,z_0),
$$  
and
$$
\bar{\mathcal P}^{2,-} u (t_0,z_0)=-\bar{\mathcal P}^{2,+}(-u)(t_0,z_0).
$$ 

\begin{Prop}\label{prop def vis}
	\begin{itemize}
		\item [(i)] An upper semi-continuous function $u : \Omega_T \longrightarrow \R$ is a subsolution to the parabolic equation  \eqref{PHE def} iff for all $(t_0, z_0) \in \Omega_T$ and $(\tau, p, Q) \in \mathcal{P}^{2,+}u(t_0, z_0) $ we have
		$$e^{\tau + G(t_0, z_0, u(t_0, z_0))}g(z_0) \leq F(HQ),$$
		where $HQ:= (H \left\langle Qz, z \right\rangle) , z \in \mathbb{C}^n=\mathbb{R}^{2n}$.
		
		\item [(ii)] A lower semi-continuous function $v : \Omega_T \longrightarrow \R$  is 
		a supersolution to the parabolic equation \eqref{PHE def}  iff for all $(t_0, z_0) \in \Omega_T$ and $(\tau, p, Q) \in \mathcal{P}^{2,-}u(t_0, z_0) $ we have
		$$e^{\tau + G(t_0, z_0, v(t_0, z_0))}g(z_0) \geq F(HQ).$$
	\end{itemize}
\end{Prop}
 The parabolic Jensen-Ishii’s maximum principle is stated as follows:
\begin{The}\cite[Theorem 8.3]{CIL92}\label{the maximal}  
		Let  $u\in USC(\Omega_T)$ and $v\in LSC(\Omega_T)$.
	Let $\phi$ be a function defined in $(0, T) \times \Omega^2$ such that $(t, \xi ,\eta) \longmapsto \phi (t,\xi, \eta)$ is continuously differentiable in $t$ and twice continuously differentiable in $(\xi ,\eta)$.  
	
	Assume that the function $(t,\xi, \eta) \longmapsto u (t,\xi) - v(t,\eta) - \phi (t,\xi, \eta)$ has a local maximum at some point $(\hat t, \hat\xi, \hat\eta) \in (0, T) \times \Omega^2$. 
	
	Assume furthermore that both $w=u$ and $w=-v$ satisfy:
	\begin{displaymath}
	(\label{Cond}\ref{Cond})\left\{
	\begin{array}{ll}
		\forall (s,z) \in \Omega & \exists r>0 \ \text{such that} \ \forall M >0 \  \exists C \ \text{satisfying} \\
		&\left.
		\begin{array}{l}
			|(t, \xi)-(s,z)| \le r,\\
			(\tau,p,Q)\in \mathcal{P}^{2,+}w(t, \xi) \\
			|w(t, \xi)|+|p| + |Q| \le M
		\end{array} \right\} \Longrightarrow \tau\le C.
	\end{array}\right.
\end{displaymath}
	
	Then for any $\kappa > 0$, there exists 
$(\tau_1,p_1,Q^+) \in \bar{\mathcal P}^{2,+} u (\hat t, \hat \xi)$, 
$(\tau_2,p_2,Q^-) \in \bar{\mathcal P}^{2,-} v (\hat t, \hat \eta)$ such that
$$\tau_1 = \tau_2 + D_t \phi (\hat t, \hat \xi,\hat \eta), \ p_1 = D_{\xi} \phi (\hat t, \hat\xi,\hat\eta),
\ p_2 = - D_{\eta} \phi (\hat t, \hat\xi,\hat\eta)$$ and
$$
-\left(\frac{1}{\kappa} + \| A \| \right) I \leq 
\left(
\begin{array}{cc}
	Q^+ &0 \\
	0 & - Q^-
\end{array}
\right) \leq  A  + \kappa A^2,
$$
where $A := D_{\xi, \eta}^2 \phi (\hat t, \hat\xi,\hat\eta)\in\mathcal{S}_{4n}$. 
\end{The}
\subsection{$\Gamma$-subharmonic functions and strictly $\Gamma$-pseudoconvex domains}
\begin{Def} Assume $U\subset \C^n$ is a domain. A  function smooth $u$ in $U$ is called  $\Gamma$-subharmonic if
$Hu(z)\in \overline{M(\Gamma, n)}$ for every $z\in U$. An upper semicontinuous function $u$ in $U$ is called  $\Gamma$-subharmonic if for every open set $V\subset U$ there exists a decreasing sequence $\{u_j\}$ of smooth
$\Gamma$-subharmonic functions in $V$ such that $u_j\rightarrow u$ as $j\rightarrow\infty$.
\end{Def}
Since $\Gamma\subset\Gamma_1$,  every $\Gamma$-subharmonic function is subharmonic. For every $1\leq k\leq n$,
 the set of $\Gamma_k$-subharmonic functions coincides with the set of $k$-subharmonic functions.
 
 The $\Gamma$-subharmonicity can be characterized through viscosity concepts. Given a continuous function $\psi\geq 0$ in $U$, we say that a function $u\in USC(U)$ is a viscosity subsolution of the
 equation $F(Hw)=\psi$ on $U$ if for every $ z_0\in\Omega$ and for every upper test function $q$ of $u$ at $z_0$, we have $F(Hq(z_0))  \geq \psi (z_0)$ (and then $Hq(z_0)\in \overline{M(\Gamma, n)}$).
 
 By \cite[Lemma 4.6, Remark 4.9 and Theorem B.8]{HL09}  (see also \cite{DDT} and \cite{DN21}), we have:
 \begin{Prop}\label{prop vis subharmonic}
 For every $u\in USC(U)$, the following conditions are equivalent:
 \begin{itemize}
 	\item[(i)] $F(Hu)\geq 0$ in the viscosity sense in $U$;
 	\item[(ii)] $u$ is $\Gamma$-subharmonic in $U$. 
 \end{itemize}
\end{Prop}
Given a smooth domain $\Omega\subset\C^n$, we say a function $u\in C^{\infty}(\overline{\Omega})$ is strictly $\Gamma$-subharmonic if $Hu(z)\in M(\Gamma, n)$ for every $z\in\overline{\Omega}$. In particular, there exists $\epsilon>0$ such that $u-\epsilon|z|^2$ is  $\Gamma$-subharmonic. 
\begin{Def}
A smooth domain $\Omega\subset\C^n$ is called strictly $\Gamma$-pseudoconvex iff there exists a 
	strictly  $\Gamma$-subharmonic function $u\in C^{\infty}(\overline{\Omega})$ satisfying $u|_{\partial\Omega}=0$, $u|_{\Omega}<0$ and
	$\nabla u(z)\neq 0$ for every $z\in\partial\Omega$. 
\end{Def}

\section{Regularizing in time}\label{sec reg}
Let $u :\Omega_T \rightarrow \mathbb{R}$ be a bounded function and let $A>osc_{\Omega_T}u$. We define
$$u^k(t, z)=\sup\{u(t+s, z)-k\left| s\right| : \left| s\right|  \leq \dfrac{A}{k}\},$$
and
$$u_k(t, z)=\inf\{u(t+s, z)+k\left| s\right| : \left| s\right|  \leq \dfrac{A}{k}\},$$
for every $k > \dfrac{2A}{T}$ and $(t, z) \in \left( \dfrac{A}{k}, T-\dfrac{A}{k}\right) \times\Omega$.

The following result generalizes \cite[Lemma 3.5]{EGZ15b}:
\begin{Lem}\label{lem regularization sub} Assume $u$ is a bounded upper semicontinuous function on $\Omega_T$. Then
	\begin{itemize}
		\item [(i)] $u^k$ is upper semicontinuous on $\left( \dfrac{A}{k}, T-\dfrac{A}{k}\right) \times \Omega$;
		\item [(ii)] for every $(t, z) \in \left( \dfrac{A}{k}, T-\dfrac{A}{k}\right) \times \Omega$, 
		$$u(t, z) \leq u^k(t, z) \leq \sup\limits_{\left| s\right| \leq A/k}u(t+s, z);$$
		\item [(iii)] if $(t, z) \in \left( \dfrac{A}{k}, T-\dfrac{A}{k}\right) \times \Omega$ and $(t+s, z) \in \left( \dfrac{A}{k}, T-\dfrac{A}{k}\right) \times \Omega$ then
		$$\left| u^k(t,z)-u^k(t+s, z)\right| \leq k\left| s\right|;$$
		\item [(iv)] if $F(Hu)\geq e^{\partial_tu+F(t, z, u)g(t,z)}$ in $\Omega_T$ then, for $k\gg 1$,
		\begin{equation}\label{eq0 lem regu}
			F(Hu^k)\geq e^{\partial_t u^k+G_k(t, z, u^k)}g(z),
		\end{equation}
		in the viscosity sense in $\left( \dfrac{A}{k}, T-\dfrac{A}{k}\right) \times \Omega$, where $$G_k(t,z,r)=\inf\limits_{\left| s\right| \leq A/k}G(t+s, z, r).$$
	\end{itemize}
\end{Lem}

\begin{proof}
	\begin{itemize}
		\item [(i)] Let $(t_0, z_0) \in \left( \dfrac{A}{k}, T-\dfrac{A}{k}\right) \times \Omega$. We will show that 
		$$u^k(t_0,z_0) \geq \limsup\limits_{(t, z) \rightarrow (t_0,z_0)}u^k(t,z).$$
		Choose $(t_m, z_m) \in \left( \dfrac{A}{k}, T-\dfrac{A}{k}\right) \times \Omega$ such that $(t_m, z_m) \rightarrow (t_0, z_0)$ as $m \rightarrow \infty$ and
		$$\limsup\limits_{(t, z) \rightarrow (t_0,z_0)}u^k(t,z)=\lim\limits_{m \rightarrow \infty} u^k(t_m,z_m).$$
		Since $u$ is upper continuous, there exists $-A/k\leq s_m\leq A/k$ such that
		$$u^k(t_m,z_m)=u(t_m+s_m,  z_m)-k\left| s_m\right|.$$
	Let $\{s_{m_l}\}$ be a subsequence of $\{s_m\}$, $s_{m_l}$ converging to a point $s_0 \in \left[ -\dfrac{A}{k}, \dfrac{A}{k}\right] $ as $m_l \rightarrow \infty$. Then
		\begin{equation*}
			\begin{aligned}
				\limsup\limits_{(t, z) \rightarrow (t_0,z_0)}u^k(t,z) & = \lim\limits_{m_l \rightarrow \infty}u^k(t_{m_l}, z_{m_l})\\
				& =\lim\limits_{m_l \rightarrow \infty} u(s_{m_l}+ t_{m_l},  z_{m_l}) -k\left| s_0\right| \\
				& \leq u(s_0+t_0, z_0)-k\left| s_0\right| \\
				& \leq u^k(t_0, z_0).\\
			\end{aligned}
		\end{equation*}
		Thus $u^k$ is upper semicontinuous in $\left( \dfrac{A}{k}, T-\dfrac{A}{k}\right) \times \Omega$.
		
		\item [(ii)] 	for every $(t, z) \in \left( \dfrac{A}{k}, T-\dfrac{A}{k}\right)\times\Omega $, we have
		$$u(t+s, z) -k|s| \leq u(t+s, z).$$
		Taking the supremum of both sides over all $s$ such that $|s| \leq \dfrac{A}{k}$, we have
		 $$u^k(t, z) \leq \sup\limits_{\left| s\right| \leq A/k}u(t+s, z).$$
		Moreover
		\begin{equation}
			u^k(t, z)=\sup\{u(t+s, z)-k\left| s\right| : \left| s\right|  \leq \dfrac{A}{k}\} 
			\geq  u(t, z).
		\end{equation}
	Then
		$$u(t, z) \leq u^k(t, z) \leq \sup\limits_{\left| s\right| \leq A/k}u(t+s, z).$$
		\item [(iii)] Assume $s_0 \in \left[ -\dfrac{A}{k}, \dfrac{A}{k}\right] $ satisfies  $u^k(t, z)=u(t+s_0, z)-k|s_0|$.
		
		If  $|s-s_0|>A/k$ then
		\begin{equation*}
			\begin{aligned}
				u^k(t, z)=u(t+s_0, z)-k|s_0| &\leq u(t+s, z)+2 osc_{\Omega_T}u-k|s_0|\\
				&\leq u^k(t+s, z)+A-k|s_0|\\
				&\leq  u^k(t+s, z)+k|s-s_0|-k|s_0|\\
				&\leq  u^k(t+s, z)+k|s|.
			\end{aligned}
		\end{equation*}
		If  $|s-s_0|\leq A/k$ then
		\begin{equation*}
			\begin{aligned}
				u^k(t, z)=u(t+s_0, z)-k|s_0|&\leq u^k(t+s, z)+k|s-s_0|-k|s_0|\\
				&\leq  u^k(t+s, z)+k|s|.
			\end{aligned}
		\end{equation*}
		Then
		$$u^k(t, z)-u^k(t+s, z)\leq k|s|.$$
		By the same way, we have 
		$$u^k(t+s, z)-u^k(t, z)\leq k|s|.$$
		Hence
		$$|u^k(t, z)-u^k(t+s, z)|\leq k|s|.$$
		
		\item [(iv)] 
		
		Let $(t_0, z_0)\in (\delta, T-\delta)\times\Omega$, $s_0\in (-A/k, A/k)$ and let $q$ be an upper test function of
		$u_{s_0}(t, z):=u(t+s_0, z)-k|s_0|$ at $(t_0, z_0)$. Then $\hat{q}(t, z):=q(t-s_0,  z)+k|s_0|$
		is an upper test function of  $u$ at $(\hat{t}, \hat{z})=(t_0+s_0,  z_0)$. Since $F(Hu)\geq e^{\dt u+G(t, z, u)}g(z)$ in the viscosity sense, we have
		\begin{equation}\label{eq2 lem regu}
			F(\hat{q}(\hat{t}, \xi))|_{\xi=\hat{z}}\geq 
			e^{\dt\hat{q}(\hat{t}, \hat{z})+G(\hat{t}, \hat{z}, \hat{q}(\hat{t}, \hat{z}))}g(\hat{z}).
		\end{equation}
		Since $\dt\hat{q}(\hat{t}, \hat{z})=\dt q(t_0, z_0)$ and
		\begin{center}
			$F(Hq(t_0, \xi))|_{\xi=z_0}=F(H\hat{q}(\hat{t}, \xi))|_{\xi=\hat{z}},$
		\end{center}
		it follows from \eqref{eq2 lem regu} that
		\begin{flushleft}
			$\begin{array}{ll}
				F(Hq(t_0, \xi))^n|_{\xi=z_0}&\geq e^{\dt q(t_0, z_0)+G(\hat{t}, \hat{z}, \hat{q}(\hat{t}, \hat{z}))}g(\hat{z})\\
				&\geq e^{\dt q(t_0, z_0)+G(\hat{t}, \hat{z}, q(t_0, z_0))}g(\hat{z})\\
				&\geq e^{\dt q(t_0, z_0)+G_k(t_0, z_0, q(t_0, z_0))}g(z_0).
			\end{array}$
		\end{flushleft}
		Hence, $u_{s_0}$ is a viscosity subsolution of the equation
		\begin{equation}\label{eq3 lem regu}
			F(Hw)^n= e^{\dt w+G_k(t, z, w)}g(z)
		\end{equation}
		in
		$(\delta, T-\delta)\times\Omega$. By Lemma \ref{lem inf sup}, we get
		$$u^k=\sup_{|s_0|\leq A/k}u_{s_0}=(\sup_{|s_0|\leq A/k}u_{s_0})^*$$ 
		is a viscosity subsolution of the equation
		 \eqref{eq3 lem regu} in
		$\left( \dfrac{A}{k}, T-\dfrac{A}{k}\right) \times \Omega$.	
	\end{itemize}
\end{proof}

Similar, we have the following lemma

\begin{Lem}\label{lem regularization super}
	Assume $u$ is a bounded lower semicontinuous function in  $\Omega_T$. Then
	\begin{itemize}
		\item[(i)] $u_k$ is a semicontinuous function in $(A/k, T-A/k)\times\Omega$.
		\item[(ii)] for every $(t, z)\in (A/k, T-A/k)\times U$, 
		\begin{center}
			$u(t, z)\geq u_k(t, z)\geq \inf_{|s|\leq A/k}u(t+s, z);$
		\end{center}
		\item[(iii)] if $(t, z)$ and $(t+s, z)$ belong in $(A/k, T-A/k)\times\Omega$ then
		\begin{center}
			$|u_k(t, z)-u_k(t+s, z)|\leq k|s|;$
		\end{center}
		\item[(iv)] if $F(Hu)\leq e^{\dt u+F(t, z, u)}g(t, z)$  in the viscosity sense in  $\Omega_T$ then,
		for  $k\gg 1$, 
		$$F(Hu_k)\leq e^{\dt u_k+G^k(t, z, u_k)}g(z),$$ in the viscosity sense in  $\left( \dfrac{A}{k}, T-\dfrac{A}{k}\right) \times \Omega$,
		where $G^k(t, z, r)=\sup_{|s|\leq A/k}G(t+s, z, r)$.
	\end{itemize}
\end{Lem}
\section{$\Gamma$-subharmonicity of viscosity subsolutions}\label{sec subharmonic}
In the case of complex Parabolic Monge-Amp\`ere equations, Eyssidieux-Guedj-Zeriahi \cite{EGZ15b} show that if $u$ is a viscosity subsolution then $u(t, z)$ is psh in $z$ for every $t$. In this section, we extend this
result to the case of the equation
\begin{center}
	$F(Hu)=e^{\dt u+G(t, z, u)}g(z).$
\end{center}
 First we recall some basic concepts. A function $u : U \subset \R^N \longrightarrow \R$ is semi-convex in $U$ if for every ball  $B \Subset U$, there exists $A > 0$ such that the function $x \longmapsto u (x) + A \vert x\vert^2$ is convex in $B$.

A second order super jet  ${\mathcal J}^{2,+} u (x_0)$ at
$x_0 \in U$ of a function $u : U \longrightarrow \R$ is the set of all $(p,Q) \in \R^N \times \mathcal S_N$ satisfying
$$
u (x) \leq u (x_0) + \langle p,x - x_0\rangle + \frac{1}{2} \langle Q (x-x_0),x - x_0\rangle +  o (\vert x - x_0\vert^2).
$$
The set $\bar{\mathcal J}^{2,+} u (x_0)$  is defined by the same way as in the parabolic case (see Preliminaries). We also denote
\begin{center}
	${\mathcal J}^{2,-} u (x_0)=-{\mathcal J}^{2,+} (-u) (x_0)$
	and $\bar{\mathcal J}^{2,-} u (x_0)=-\bar{\mathcal J}^{2,+} (-u) (x_0)$.
\end{center}
The following lemma is a special case of \cite[Lemma 3.2]{EGZ15b}:
\begin{Lem} \label{lem:subpp} Let $\psi\geq 0$ be a continuous function in a a domain
	$\Omega\Subset\C^n$.
	1.  Assume $u$ is a semi-convex function in $\Omega$ satisfying
	$$
	F(HQ) \geq \psi, \forall (p,Q) \in  {\mathcal J}^{2,+} u (z_0),
	$$
	for every $z_0\in\Omega\setminus N$, where $N$ is a null set.
	Then $F(Hu)\geq \psi$ in the viscosity sense in $\Omega$. 
	
	2. Assume $v$ is a semi-concave $\Omega$ satisfying, for almost all $z_0\in\Omega$, for every
	$(p,Q) \in  {\mathcal J}^{2,-} w (x_0)$, if $HQ\in\overline{M(\Gamma, n)}$ then
$F(HQ)\leq \psi$. Then $F(Hv)\leq \psi$ in the viscosity sense $\Omega$.
	\end{Lem}
By the same method as in \cite[Proposition 3.6]{EGZ15b}, we get the following result:
\begin{Prop} \label{lem:PartialSol}
	Let $H$ be a continuous function in $(0, T)\times\Omega\times\R$ and let  $h$
	be a non-negative continuous function in $\Omega_T$.
	Assume that $w : \Omega_T \longrightarrow \R$  is a subsolution (resp., supersolution) of the equation
\begin{equation}\label{eq0 prop partial}
	e^{ H (t,z,w)} h(t, z) - F(Hw) = 0,
\end{equation}
	in the viscosity sense in $\Omega_T$. Then, for every $t_0\in (0, T)$, the function 
	$z \longmapsto w (t_0,z)$ 
	is a subsolution (resp., supersolution)
	$e^{ H (t_0, z, \psi)} h(t_0, z) - F(H \psi)= 0$ in $\Omega$. 
\end{Prop}

\begin{proof} 
	We give the proof for the case where $w$ is a subsolution. The proof for the case of supersolutions is similar. By using approximations, we can also assume that $w$ is bounded.
	
	Let $A>osc_{\Omega_T}w$. For every $\e>0$, $\e<t<T-\e$ and $z\in\Omega(\epsilon):= \{z\in \Omega \ | \ d(z,\partial \Omega)> \e\}$, we define
	by $w_\e$ the sup-convolution:
	$$w_\e(t, z)=\sup\left\{w(s, \xi)-\dfrac{A}{\epsilon^2}\left(|t-s|^2+|z-\xi|^2\right):
	(s, \xi)\in\Omega_T\right\}.$$
	
	Then the function $v := w_\e$ satisfies
	\begin{equation}\label{eq1 prop partial}
		F(Hv) \geq e^{ H_\e (t,z,v)} h_\e,
	\end{equation}
in the viscosity sense in $U(\epsilon):=(\e, T-\e)\times\Omega(\epsilon)$, where
	$$
	h_\e (t,z) := \inf \{ h (t',z') ;  \vert t' - t\vert , \vert z' - z\vert \leq  \e\}
	$$ 
	and 
		$$
	H_\e (t,z) := \inf \{ H (t',z') ;  \vert t' - t\vert , \vert z' - z\vert \leq  \e\}
	$$ 
	
	Since $w_\e$ is semi-convex in $U(\epsilon)$, it follows from Alexandrov's theorem \cite{Ale39} that it is  twice differentiable almost everywhere in $U(\epsilon)$. Then the inequality \eqref{eq1 prop partial} is satisfied pointwise almost everywhere.  Moreover, by Fubini Theorem, for almost all  $t_0 \in (\e,T-\e)$, there exists a null set $E^{t_0} \subset \Omega(\epsilon)$ such that, for every $z_0 \notin E^{t_0}$, the function  $w_\e $ is  twice differentiable at $(t_0,z_0)$.
	By the definition, we have $\mathcal J^2 w_\e (t_0,z_0) = \{(\tau,p,\kappa,Q)\}$ and $\{(p,Q)\} = {\mathcal J}^2 \p (z_0)$, where $\p = w_\e (t_0,\cdot)$. 
	Using the inequality \eqref{eq1 prop partial} at $(t_0,z_0)$, we get
	$$ 
	F(HQ) \geq  e^{ H_\e (t_0,z_0, w_\e(t_0,z_0))} h_\e (t_0,z_0) .
	$$
	
	Hence, for almost all $t_0 \in (0, T)$, the function $\p (z) := w_\e (t_0,z)$
	is  twice differentiable at almost every $z_0 \in \Omega_\e$ and satisfies
	$$
	F(H\p(z_0)) \geq  e^{ H_\e (t_0,z_0, \p (z_0))} h_\e (t_0,z_0).
	$$
	Therefore, it follows from Lemma \ref{lem:subpp} that
	 $F(H\p) \geq e^{ H_\e (t_0,\cdot, \p)} h_\e (t_0,\cdot)$ 
in the viscosity sense in  $\Omega(\epsilon)$.
	Since $h^\e \to h$ and $G^\e \to G$ locally uniformly in $U$, it follows from Lemma \ref{lem inf sup} that
	 $w (t_0,\cdot) = \lim_{\e \to 0} w_\e (t_0,\cdot)$ is a viscosity subsolution for the equation
	$$
F(H\p)=	e^{ H (t_0,\cdot, \p)}  h(t_0, z)
	$$ 
	in $\Omega$. 
	This is true for every $t_0 \in (0, T)\setminus I$, where $I\subset (0, T)$ is a null set.
	Moreover, this holds true for the case  $t_0\in I$ by using approximations  and applying the elliptic version of Lemma  \ref{lem inf sup} (see \cite{CIL92}).
Letting $\epsilon\searrow 0$, we obtain
 	$$F(H w(t_0, z))\geq e^{ H (t_0,z, w(t_0, z))} h(t_0, z),$$
 in the viscosity sense in $\Omega$. for every $t_0\in (0, T)$.
 	
 	This finishes the proof.
\end{proof}
By Lemma \ref{lem:PartialSol} and Proposition \ref{prop vis subharmonic}, we have the following result about
the $\Gamma$-subharmonicity of viscosity subsolutions for the parabolic equation \eqref{PHE def}:
\begin{Cor} \label{SPSH} 
	Assume $u$ is a viscosity subsolution for $F(Hw)=e^{\dt w+G(t, z, w)}g(z)$ in $\Omega_T$.  Then, for every $t_0\in (0, T)$, the function $u(t_0, z)$ is $\Gamma$-sh in
	$\Omega$.
\end{Cor}
\section{Comparison principle and Perron method}\label{sec compa}
\subsection{Comparison principle}
\begin{Lem}\label{lem compa}
	Assume that $f$ satisfies
	\begin{equation}\label{eq f lem compa}
		\lim\limits_{R\to\infty}f(R, R, ...,R)=\infty.
	\end{equation}
	Suppose $u\in USC\cap L^{\infty} ([0, T)\times\overline{\Omega})$ 
	and $v\in LSC\cap L^{\infty} ([0, T)\times\overline{\Omega})$  is, respectively, a viscosity subsolution
	and supersolution for the equation
	\begin{equation}\label{PMA lem compa}
		F(H w)=e^{\dt w+F(t, z, w)}g(z),
	\end{equation}
	in $\Omega_T$ such that $u$ and $v$ are locally Lipschitz in $t$. Then
	\begin{equation}\label{eq0 lem compa}
		\sup\limits_{\Omega_T}(u-v)\leq \sup\limits_{\partial_P\Omega_T}(u-v)_+,
	\end{equation}
where $\partial_P\Omega_T=(\{0\}\times\overline{\Omega})\cup((0, T)\times\partial\Omega)$.
\end{Lem}
\begin{proof}
	Let $\delta>0$ be an arbitrary positive constant. Denote
	\begin{center}
		$h(t, z)=u(t, z)-v(t, z)-\dfrac{\delta}{T-t}+\delta(|z|^2-C),$
	\end{center}
	where $C=\sup\limits_{z\in\Omega}|z|^2$.
	We will show that
	\begin{equation}\label{eq1 lem compa}
		\max\limits_{\overline{\Omega_T}}h\leq 
		\max\limits_{\partial_P\Omega_T}h_+.
	\end{equation}
	Assume that \eqref{eq1 lem compa} is false. Then, there exist $(t_0, z_0)\in\Omega_T$ such that
	\begin{center}
		$M:=h(t_0, z_0)=\max\limits_{[0, T)\times\overline{\Omega}}h>
		\max\limits_{\partial_P\Omega_T}h_+.$
	\end{center}
	For each $N>0$, we define
	\begin{center}
		$h_N(t, \xi, \eta)=u(t, \xi)+\delta(|z|^2-C)-v(t, \eta)-\dfrac{\delta}{T-t}-\dfrac{N|\xi-\eta|^2}{2},$
	\end{center}
	and choose $(t_N, \xi_N, \eta_N)\in [0, T)\times\overline{\Omega}^2$ such that
	\begin{center}
		$h_N(t_N, \xi_N, \eta_N)=\max\limits_{[0, T)\times\overline{\Omega}^2}h_N=:M_N.$ 
	\end{center}
	It follows from \cite[Proposition 3.7]{CIL92} that $\lim_{N\to\infty}N|\xi_N-\eta_N|^2=0$ and we can assume
	$\xi_N, \eta_N\rightarrow z_0$, $t_N\rightarrow t_0$ as $N\rightarrow\infty$. In particular, there exists $N_0>0$ such that $(t_N, \xi_N, \eta_N)\in (0, T)\times\Omega^2$ for every $N\geq N_0$.
	
	By the Lipschitz continuity in $t$ of $u$ and $v$, the functions $w_1=u+\delta(|z|^2-C)$ and  $w_2=-v$ satisfy the condition \eqref{Cond} in Theorem \ref{the maximal}. Hence, it follows from Theorem \ref{the maximal} that, for every $N>N_0$, 
	there are $(\tau_{N1}, p_{N1}, Q_N^+)\in \overline{\mathcal{P}}^{2,+}w_1(t_N, \xi_N)$ and
	$(\tau_{N2}, p_{N2}, Q_N^-)\in \overline{\mathcal{P}}^{2,-}v(t_N, \eta_N)$ such that
	\begin{equation}\label{eq2 lem compa}
		\tau_{N1}=\tau_{N2}+\dfrac{\delta}{(T-t_N)^2}, 
	\end{equation}
	and $Q_N^-\geq Q_N^+$ (i.e., $\langle Q_N^+\zeta, \zeta\rangle\geq \langle Q_N^-\zeta, \zeta\rangle$
	for all $\zeta\in\R^{2n}$). In particular, we have
	\begin{equation}\label{eq3 lem compa}
		H Q_N^-\geq H Q_N^+\geq \delta.Id>0,
	\end{equation}
 where the second inequality holds due to Proposition \ref{prop def vis}.
	Moreover, it follows from Proposition \ref{prop def vis} that
	\begin{equation}\label{eq4 lem compa}
		e^ {\tau_{N1} +  G (t_N, \xi_N, u (t_N, \xi_N))} g (\xi_N)  \leq 
		F(HQ_N^+-\delta.Id).
	\end{equation}
	and
	\begin{equation}\label{eq5 lem compa}
		e^{\tau_{N2} +  G (t_N, \eta_N, v (t_N ,\eta_N))} g (\eta_N)  \geq F(H Q_N^-).
	\end{equation}
Since $u$ and $v$ are locally Lipschitz in $t$, there exists $C_1>0$ such that
\begin{equation}\label{eq6 lem compa}
	|\tau_{N1}|, |\tau_{N2}|\leq C_1,
\end{equation}
for every $N\gg 1$. Then, by \eqref{eq5 lem compa} and by the continuity of $G, g$, there exists
$C_2>0$ such that
\begin{equation}\label{eq7 lem compa}
	F(H Q_N^-)\leq C_2,
\end{equation}
for all $N\gg 1$.

By the condition \eqref{eq f lem compa}, there exists $R>0$ such that $F(A)\geq C_2+1$ for every $A\geq R.Id$.
Therefore, by \eqref{eq3 lem compa}, \eqref{eq7 lem compa} and by the concavity $F$,
 we get
\begin{equation}\label{eq8 lem compa}
	F(HQ_N^+)\geq \dfrac{R.F(HQ_N^+-\delta. Id)}{R+\delta}
	+\dfrac{\delta.F(HQ_N^++R.Id)}{R+\delta}
	\geq F(HQ_N^+-\delta. Id)+\dfrac{\delta}{R+\delta}.
\end{equation}
Combining \eqref{eq3 lem compa}, \eqref{eq4 lem compa}, \eqref{eq5 lem compa} and
\eqref{eq8 lem compa}, we have
\begin{equation}\label{eq9 lem compa}
		e^ {\tau_{N1} +  G (t_N, \xi_N, u (t_N, \xi_N))} g (\xi_N)+\dfrac{\delta}{R+\delta}  \leq
		e^{\tau_{N2} +  G (t_N, \eta_N, u (t_N ,\eta_N))} g (\eta_N).
\end{equation}
Combining \eqref{eq2 lem compa}, \eqref{eq6 lem compa} and \eqref{eq9 lem compa}, we obtain
\begin{center}
	$\dfrac{e^{-C_1}\delta}{R+\delta}\leq 
	e^{G (t_N, \eta_N, v (t_N ,\eta_N))} g (\eta_N)
	-e^ {\frac{\delta}{(T-t_N)^2} +  G (t_N, \xi_N, u (t_N, \xi_N))} g (\xi_N).$
\end{center}
Letting $N\to\infty$, we get
\begin{equation}\label{eq10 lem compa}
	\dfrac{e^{-C_1}\delta}{R+\delta}\leq e^{G (t_0, z_0, v (t_0, z_0))} g (z_0)
	-e^ {\frac{\delta}{(T-t_0)^2} +  G (t_0, z_0, u (t_0, z_0))} g (z_0).
\end{equation}
On the other hand, by the condition $M=h(t_0, z_0)>0$, we have $u(t_0, z_0)\geq v(t_0, z_0)$. Since $G$ is
non-decreasing in the last variable, it follows that the RHS of \eqref{eq10 lem compa} is non-positive. This is a contradiction.
Thus, \eqref{eq1 lem compa} is true. Using \eqref{eq1 lem compa} and letting $\delta\to 0$, we obtain \eqref{eq0 lem compa}.
\end{proof}
Using the method as in \cite{EGZ15b} and applying the lemmas
 \ref{lem regularization sub}, \ref{lem regularization super},
\ref{lem compa}, we obtain the following comparison principle:
\begin{The}\label{the compa}
	Assume  $f$ satisfies
\begin{equation}\label{eq f  compa}
	\lim\limits_{R\to\infty}f(R, R, ...,R)=\infty.
\end{equation}
Suppose $u\in USC\cap L^{\infty} ([0, T)\times\overline{\Omega})$ 
and $v\in LSC\cap L^{\infty} ([0, T)\times\overline{\Omega})$ is, respectively, a viscosity subsolution and supersolution of the equation
\begin{equation}\label{PHE compa}
	F(H w)=e^{\dt w+F(t, z, w)}g(z),
\end{equation}
in $\Omega_T$. Then
\begin{center}
	$\sup\limits_{\Omega_T}(u-v)\leq \sup\limits_{\partial_P\Omega_T}(u-v)_+.$
\end{center}
\end{The}
\begin{proof}
	Without loss of generality, we can assume that $u\leq v$ in $\partial_P\Omega_T$.
	Let $\epsilon>0$ and $0<S<T$. Since $u(t_1, z)-v(t_2, z)$ is upper semi-continuous in
	$[0, T)\times [0, T)\times\overline{\Omega}$, there exists $0<\delta<\min\{S/2, T-S\}$ such that
	\begin{equation}\label{eq1 the compa}
		u(t_1, z)-v(t_2, z)\leq \sup\limits_{\partial_P\Omega_T}(u-v)_++\epsilon,
	\end{equation}
for every $(t_1, z), t_2, z)\in\Omega_{2\delta}\cup\partial_P\Omega_{S+\delta}$
with $|t_1-t_2|<2\delta$. Denote $A=2\sup_{\Omega_T}(|u|+|v|)+1$. Since $G$ is continuous, we can choose $\delta$ such that
\begin{equation}\label{eq2 the compa}
	|G(t_1, z, r)-G(t_2, z, r)|<\epsilon,
\end{equation}
for every $(t_1, t_2, z, r)\in [0, T]\times [0, T]\times\overline{\Omega}\times [-A, A]$
with $|t_1-t_2|<2\delta$.
 
 For all $k>2A/\delta$, we define
$$u^k(t, z)=\sup\{u(t+s, z)-k\left| s\right| : \left| s\right|  \leq \dfrac{A}{k}\},$$
and
$$v_k(t, z)=\inf\{v(t+s, z)+k\left| s\right| : \left| s\right|  \leq \dfrac{A}{k}\},$$
for every $(t, z) \in ( \delta, S) \times\Omega$. By Lemma \ref{lem regularization sub}, Lemma
\ref{lem regularization super} and by \eqref{eq2 the compa}, for each $k>2A/\delta$,
we have $u^k(t, z)-\epsilon t$ and $v_k+\epsilon t$, is, respectively, a subsolution and a supersolution
for the equation
$$F(H w)=e^{\dt w+G(t, z, w)}g(z),$$
in $( \delta, S) \times\Omega$. Moreover, by \eqref{eq1 the compa}, we have
$$(u^k(t, z)-\epsilon t)-(v_k+\epsilon t)\leq (2S+1)\epsilon,$$
in $\partial_P(( \delta, S) \times\Omega)$. Therefore, it follows from Lemma \ref{lem compa} that
$$(u^k(t, z)-\epsilon t)\leq (v_k+\epsilon t)+(2S+1)\epsilon,$$
in $( \delta, S) \times\Omega$. Letting $k\to\infty$, we get
$$(u(t, z)-\epsilon t)\leq (v+\epsilon t)+(2S+1)\epsilon,$$
in $( \delta, S) \times\Omega$. Letting $\epsilon\searrow 0$ and $S\nearrow T$, we obtain $u\leq v$
in $\Omega_T$.

The proof is completed.
\end{proof}
\subsection{Perron method}
A function $u\in USC([0, T)\times\bar{\Omega})$ is called $\epsilon$-subbarrier for
\eqref{PHE dirichlet} if $u$ is subsolution to $F(Hw)=e^{\dt w+G(t,z,w)}g(z)$
in the viscosity sense such that $u_0-\epsilon\leq u_*\leq u\leq u_0$ in 
$\{0\}\times\bar{\Omega}$ and $\varphi-\epsilon\leq u_*\leq u\leq \varphi$ in
$[0, T)\times\partial\Omega.$

Similar, a function $v\in LSC([0, T)\times\bar{\Omega})$ is called $\epsilon$-superbarrier for
\eqref{PHE dirichlet} if $v$ is supersolution to $F(Hw)=e^{\dt w+G(t,z,w)}g(z)$
in the viscosity sense such that $u_0+\epsilon\geq v^*\geq v\geq u_0$ in 
$\{0\}\times\bar{\Omega}$ and $\varphi+\epsilon\geq v^*\geq v\geq \varphi$ in
$[0, T)\times\partial\Omega.$

By the Perron method, we have the following lemma:
 \begin{Lem}\label{lem perron}
 	Assume that for every $\epsilon>0$, the Cauchy-Dirichlet problem \eqref{PHE dirichlet} has
 	a continuous $\epsilon$-superbarrier which is Lipschitz in $t$
 	and a continuous $\epsilon$-subbarrier. 
 	Denote by $S$ the family of all continuous subsolutions  to	\eqref{PHE dirichlet}.
 	Then $\Phi_S=\sup\{v: v\in S\}$ is a
 	discontinuous viscosity solution to \eqref{PHE dirichlet}, i.e.,
 	$(\Phi_S)^*$ is a subsolution and $(\Phi_S)_*$ is a supersolution.
 \end{Lem}
 \begin{proof} 
 	We use the same arguments as in the proof of \cite[Lemma 2.12]{DLT}.
 	By the existence of $\epsilon$-subbarriers and $\epsilon$-superbarriers,
 	and by using Theorem \ref{the compa}, we have
 	$\Phi_S$ satisfies the boundary condition and initial condition of \eqref{PHE dirichlet}.
 	Moreover, the set of
 	continuous points of $\Phi_S$ contains $\partial_P\Omega_T$. Hence, $\Phi_S$ is a 
 	discontinuous viscosity solution to the Cauchy-Dirichlet problem \eqref{PHE dirichlet} iff it is a 
 	discontinuous viscosity solution to the equation
 	\begin{equation}\label{eq1 perron}
 		F(Hw)=e^{\dt w+G(t,z,w)}g(z).
 	\end{equation}
 	By Lemma \ref{lem inf sup},
 	we have $(\Phi_S)^*$ is a subsolution to \eqref{eq1 perron}. Then, it remains to show
 	that $(\Phi_S)_*(=\Phi_S)$ is a supersolution to \eqref{eq1 perron}.
 	
 	Assume that $\Phi_S$ is not a supersolution to \eqref{eq1 perron}. Then, there exist a point
 	$(t_0, z_0)\in (0, T)\times\Omega$, an open neighborhood $I\times U\Subset (0, T)\times\Omega$ of 
 	$(t_0, z_0)$ and a $C^{(1, 2)}$ function $p: I\times U\rightarrow\R$ such that
 	$(\Phi_S-p)(t_0, z_0)=\min\limits_{I\times U} (\Phi_S-p)$, $Hp(t_0, z_0)\in M(\Gamma, n)$  and
 	$F(Hp(t_0, z))\mid_{z=z_0}>e^{\partial_{t}p(t_0, z_0)+G(t_0, z_0, p(t_0, z_0))}g (z_0)$.
 	
 	Let $0<\delta\ll 1$ such that, in a neighborhood $I_{\delta}\times U_{\delta}\subset I\times U$
 	of   $(t_0, z_0)$, the functions $$p_{\delta}^{\pm}=p-\delta |z-z_0|^2\mp\delta (t-t_0)$$
 	are $\Gamma$-subharmonic in $z$ and
 	$$F(Hp_{\delta}^\pm)>e^{\partial_{t}p_{\delta}^\pm+G(t, z, p_{\delta}^\pm(t, z))}g(z).$$ Let $0<\delta_1\ll 1$ such that 
 	$$\delta_2:=-\delta_1+\delta\min\limits_{\partial (I_{\delta}\times U_{\delta})}
 	(|t-t_0|+|z-z_0|^2)>0.$$ We define
 	\begin{center}
 		$\tilde{p}:=\min\{p_{\delta}^+, p_{\delta}^-\}+\delta_1=
 		p-\delta(|t-t_0|+|z-z_0|^2)+\delta_1.$
 	\end{center}
 	Then $\tilde{p}$ is a continuous subsolution to \eqref{eq1 perron}
 	in $I_{\delta}\times U_{\delta}$
 	such that $\tilde{p}\leq \Phi_S-\delta_2$ on $\partial (I_{\delta}\times U_{\delta})$
 	and  $\tilde{p}(t_0, z_0)>\Phi_S (t_0, z_0)$.
 	
 	By the definition of $\Phi_S$, for each $a\in \partial(I_{\delta}\times U_{\delta})$,
 	there exists  $u_a\in S$ such that 
 	$$u_a(a)>\Phi_S(a)-\frac{\delta_2}{2}\geq \tilde{p}(a)+\dfrac{\delta_2}{2}.$$
 	By the continuity of $u_a$ and $\tilde{p}$, there exists a neighborhood $V_a$ of $a$
 	such that $u_a>\tilde{p}-\delta_2/2$ in $V_a$. Since 
 	$\partial(I_{\delta}\times U_{\delta})$ is compact, there exists $a_1,..., a_m\in
 	\partial(I_{\delta}\times U_{\delta})$ such that $\partial(I_{\delta}\times U_{\delta})$
 	is covered by $\{V_{a_j}\}_{j=1}^m$. Define
 	\begin{center}
 		$u=\max\{u_{a_1},..., u_{a_m}\}.$
 	\end{center}
 	Then $u\in S$ and $u>\tilde{p}+\dfrac{\delta_2}{2}$ near
 	$\partial(I_{\delta}\times U_{\delta})$. Define
 	\begin{center}
 		$\tilde{u}(t, z)=\begin{cases}
 			u(t,z)\qquad\mbox{if}\quad (t, z)\notin I_{\delta}\times U_{\delta},\\
 			\max\{u(t, z), \tilde{p}(t, z)\}\qquad\mbox{if}\quad (t, z)\in I_{\delta}\times U_{\delta}.
 		\end{cases}$
 	\end{center}
 	Then $\tilde{u}\in S$ and $\tilde{u}(t_0, z_0)\geq \tilde{p}(t_0, z_0)>\Phi_S(t_0, z_0)$.
 	We get a contradiction. 
 	
 	Thus $\Phi_S$ is a supersolution to \eqref{eq1 perron}.
 	
 	The proof is completed.
 \end{proof}
\section{The existence and uniqueness of solution}\label{sec exist}
In this section, we use the comparison principle (Theorem \ref{the compa}) and the Perron method (Lemma \ref{lem perron}) to prove
the existence and uniqueness of solution to the Cauchy-Dirichlet problem \eqref{PHE dirichlet}.
In order to use Lemma \ref{lem perron}, we first show the existence of $\epsilon$-sub/super-barriers.
\subsection{The construction of $\epsilon$-sub/super-barriers}
\begin{Prop}\label{prop.subbarrier}
	Assume $\Omega$ is strictly $\Gamma$-pseudoconvex and $f$ satisfies
	\begin{equation}\label{eq0 subbarrier}
		\lim\limits_{R\to\infty}f(R,...,R)=\infty.
	\end{equation}
 Then, for every  $\epsilon>0$,
 	there exists a continuous $\epsilon$-subbarrier for \eqref{PHE dirichlet}
 which is Lipschitz in $t$. 
\end{Prop}
\begin{proof}

	Since $\Omega$ is  strictly $\Gamma$-pseudoconvex, there exist $c>0$ and
 $\rho\in C^2(\bar{\Omega})$ such that $\rho|_{\partial\Omega}=0$,
	$\nabla\rho|_{\partial\Omega}\neq 0$
	and 
	\begin{equation}\label{eq1 subbarrier}
		H\rho(z)-I\in M(\Gamma, n),
	\end{equation}
	 for every $z\in\overline{\Omega}$ . Denote $c=\sup\limits_{\Omega}(-\rho)$. Then, there exists
	$M_1\gg 1$ such that the function
	\begin{center}
		$\underline{u}_1=u_0+\dfrac{\epsilon (\rho-c)}{2c}-M_1t,$
	\end{center}
	is a subsolution for the Cauchy-Dirichlet problem \eqref{PHE dirichlet}.
	
	Let $\varphi_{\epsilon}\in C^{\infty}(\R\times\C^n)$ such that
	\begin{center}
		$\varphi-\dfrac{\epsilon}{2} \leq\varphi_{\epsilon}\leq\varphi,$
	\end{center}
	in $[0, T]\times\partial\Omega$. By \eqref{eq0 subbarrier} and \eqref{eq1 subbarrier},
	 there exists $M_2\gg 1$ such that the function
	\begin{center}
		$\underline{u}_2=\varphi_{\epsilon}-\dfrac{\epsilon}{2}+M_2\rho$,
	\end{center}
	is  a subsolution for the Cauchy-Dirichlet problem \eqref{PHE dirichlet}.
	
	Now, we define $\underline{u}=\max\{\underline{u}_1, \underline{u}_2\}$. Then
	$\underline{u}$ is a continuous $\epsilon$-subbarrier for \eqref{PHE dirichlet}.
\end{proof}
\begin{Prop}\label{prop.superbarrier} 
		Assume $\Omega$ is a bounded smooth domain and $(u_0, g)$ is  
	admissible. Then, for every $\epsilon>0$,
	there exists a continuous $\epsilon$-superbarrier for \eqref{PHE dirichlet}
	which is Lipschitz in $t$. 
\end{Prop}
Here, we recall that the pair $(u_0, g)$ is admissible if for all $\epsilon>0$, there
exist $u_{\epsilon}\in C(\bar{\Omega})$ and $C_{\epsilon}>0$ such that $u_0\leq u_{\epsilon}\leq u_0+\epsilon$ and $F( Hu_{\epsilon})\leq 
e^{C_{\epsilon}}g(z)$ in the viscosity sense. 
\begin{proof}
	Since $(u_0(z), \mu(0, z))$ is  admissible, there	exist $u_{\epsilon}\in C(\bar{\Omega})$ and $C_{\epsilon}>0$ such that $u_0+\epsilon\leq u_{\epsilon}\leq u_0+2\epsilon$ and 
	$F(H u_{\epsilon})\leq e^{C_{\epsilon}}g(z)$ in the viscosity sense.
		Let $\varphi^{\epsilon}\in C^{\infty}(\R\times\C^n)$ such that
	\begin{center}
		$\varphi\leq\varphi^{\epsilon}\leq\varphi+\epsilon,$
	\end{center}
	in $[0, T]\times\partial\Omega$. 
	Denote
	$$M_1=\sup\{|G(t, z, u_0(z))|: (t, z)\in (0, T)\times\Omega \}+\sup_{\Omega_T}|\dt\varphi_{\epsilon}|.$$
	By using the definition,
	we have $\overline{u}_1(t, z):=u_{\epsilon}(z)+(C_{\epsilon}+M_1)t$ is a viscosity supersolution for the Cauchy-Dirichlet problem \eqref{PHE dirichlet}.
	
	For every $t\in [0, T]$, we denote by $\overline{u}_2(t, z)$ the unique solution to the equation
	\begin{equation}
		\begin{cases}
			\overline{u}_2\in C^{\infty}((0, T)\times \overline{\Omega}),\\
			\Delta\overline{u}_2(t, z)=0,\\
			\overline{u}_2(t, z)|_{\partial\Omega}=\varphi^{\epsilon}(t, z)|_{\partial\Omega}.
		\end{cases}
	\end{equation}
	
	Then $\overline{u}_2: [0, T]\times\overline{\Omega}\rightarrow\R$ is Lipschitz in $t$. In particular,
	$\overline{u}_2\in C([0, T]\times\overline{\Omega})$. It is easy to see that $\overline{u}_2$ is also
	a supersolution to \eqref{PHE dirichlet}.
	
	Now, we define $\overline{u}=\min\{\overline{u}_1, \overline{u}_2\}$. It is clear that
	$\overline{u}$ is a continuous $2\epsilon$-superbarrier for \eqref{PHE dirichlet}.
\end{proof}
\subsection{The existence and uniqueness of solution}
\begin{The}
	Assume $\Omega$ is a strictly $\Gamma$-pseudoconvex domain and $f$ satisfies
	\begin{center}
		$\lim\limits_{R\to\infty}f(R,...,R)=\infty.$
	\end{center}
	Then, the Cauchy-Dirichlet problem \eqref{PHE dirichlet} admits a unique viscosity solution iff $(u_0, g)$ is admissible.
\end{The}
\begin{proof}
	Assume $(u_0, g)$ is admissible. By Proposition \ref{prop.subbarrier}, Proposition
	\ref{prop.superbarrier} and Lemma \ref{lem perron}, we have
	 $\Phi_S=\sup\{v: v\in S\}$ is a discontinuous viscosity solution of
	\eqref{PHE dirichlet}, where $S$ is the family of all continuous subsolutions for \eqref{PHE dirichlet}. It is obvious that $(\Phi_S)^*\geq(\Phi_S)_*$. Moreover, 
	it follows from \ref{the compa} that $(\Phi_S)^*\leq(\Phi_S)_*$. Hence
	 $(\Phi_S)^*=(\Phi_S)_*$ and $\Phi_S$ is a viscosity solution of \eqref{PHE dirichlet}.
	 By Theorem \ref{the compa}, we also obtain the uniqueness of solution.
	 
	 For the converse we assume \eqref{PHE def} has a viscosity solution $u$.
	 By the parts (iii) and (iv) of Lemma \ref{lem regularization super}, we have
	 \begin{center}
	 	$F(Hu_k)\leq e^{k+C}g(z)$,
	 \end{center}
 in the viscosity sense in $(A/k, T-A/k)\times\Omega$ for every $k\gg 1$,
where $A>2osc_{\Omega_T}u$ is a constant, $C=\sup\{G(t, z, \sup u): (t, z)\in\Omega_T \}<\infty$ and
$$u_k(t, z)=\inf\{u(t+s, z)+k\left| s\right| : \left| s\right|  \leq \dfrac{A}{k}\}.$$
By Proposition \ref{lem:PartialSol}, we have
$$F(Hu_k(t_0, z))\leq e^{k+C}g(z),$$
in the viscosity sense in $\Omega$ for every $k\gg 1$ and for each $t_0\in (A/k, T-A/k)$.

For every $\epsilon>0$, by the part (ii) of Lemma \ref{lem regularization super} and by the continuity of $u$, there exist $0<\delta<T$ and $k_0\gg 1$ such that $|u(\delta, z)-u_0(z)|<\epsilon/4$ and
 $|u_{k_0}(\delta, z)-u(\delta, z)|<\epsilon/4$. Hence, denoting $u_{\epsilon}(z)=u_{k_0}(\delta, z)+\epsilon/2$, we have $u_0 \leq u_\epsilon \leq u_0 +\epsilon$ và $F(Hu_\epsilon) \leq e^{k_0+C}g(z)$ in the viscosity sense in $\Omega$. Thus $(u_0, g)$ is admissible.
 
 The proof is completed.
\end{proof}

\end{document}